\documentclass[11pt,A4paper]{article}

\usepackage[left=32mm,right=32mm,top=30mm,bottom=32mm]{geometry}
\usepackage{labelfig}
\usepackage{epsfig}
\usepackage{epstopdf}

\usepackage{xfrac}

\usepackage[percent]{overpic}

\usepackage{color}
\usepackage{amsthm,amsmath,amssymb}
\usepackage{booktabs}
\usepackage{mathpazo}
\usepackage{microtype}
\usepackage{overpic}
\usepackage{bm}
\usepackage{sectsty}
\usepackage[
	pdftitle={PDFTitle},
	pdfauthor={Hugo Parlier},
	ocgcolorlinks,
	linkcolor=linkred,
	citecolor=linkred,
	urlcolor=linkblue]
{hyperref}

\definecolor{linkred}{rgb}{0.48,0.1,0.05}
\definecolor{linkblue}{RGB}{16, 78, 139}

\usepackage[hang,flushmargin]{footmisc}
\usepackage{enumitem}
\usepackage{titlesec}
	\titlespacing{\section}{0pt}{12pt}{0pt}
	\titlespacing{\subsection}{0pt}{6pt}{0pt}
	
\titlelabel{\thetitle.\quad}

\makeatletter 

\long\def\@footnotetext#1{%
\H@@footnotetext{%
\ifHy@nesting 
\hyper@@anchor{\@currentHref}{#1}%
\else 
\Hy@raisedlink{\hyper@@anchor{\@currentHref}{\relax}}#1%
\fi 
}}

\def\@footnotemark{%
\leavevmode 
\ifhmode\edef\@x@sf{\the\spacefactor}\nobreak\fi 
\H@refstepcounter{Hfootnote}%
\hyper@makecurrent{Hfootnote}%
\hyper@linkstart{link}{\@currentHref}%
\@makefnmark 
\hyper@linkend 
\ifhmode\spacefactor\@x@sf\fi 
\relax 
}%

\ifFN@multiplefootnote%
\renewcommand*\@footnotemark{%
\leavevmode 
\ifhmode 
\edef\@x@sf{\the\spacefactor}%
\FN@mf@check 
\nobreak 
\fi 
\H@refstepcounter{Hfootnote}%
\hyper@makecurrent{Hfootnote}%
\hyper@linkstart{link}{\@currentHref}%
\@makefnmark 
\hyper@linkend 
\ifFN@pp@towrite 
\FN@pp@writetemp 
\FN@pp@towritefalse 
\fi 
\FN@mf@prepare 
\ifhmode\spacefactor\@x@sf\fi 
\relax%
}%
\fi 

\makeatother 

\theoremstyle{plain}
\newtheorem{theorem}{Theorem}[section]

\newtheorem{lemma}[theorem]{Lemma}

\theoremstyle{definition}

\newtheorem{notation}[theorem]{Notation}

\newcommand{\ii}{{\mathit i}}

\newcommand{\arcsinh}{{\,\rm arcsinh}}

\newcommand{\diam}{{\rm diam}}

\newcommand{\teich}{{\mathcal T_{g}}}
\newcommand{\moduli}{{\mathcal M_{g}}}
\newcommand{\moduliep}{{\mathcal M_g^\varepsilon}}

\newcommand{\ts}{{\mathbf X_{g}}}

\newcommand{\tri}{{\mathbf Y_{g}}}
\newcommand{\trio}{{\mathbf Y^{\,0}_{g}}}

\newcommand{\DP}{\mathrm{DP}}
\newcommand{\MDP}{\mathcal{M}\mathrm{DP}}

\sectionfont{\large \sc}
\subsectionfont{\normalsize}

\long\def\symbolfootnote[#1]#2{\begingroup%
\def\thefootnote{\fnsymbol{footnote}}\footnote[#1]{#2}\endgroup}

\def\blfootnote{\xdef\@thefnmark{}\@footnotetext}

\begin{document}

{\Large \bfseries \sc Relative shapes of thick subsets of moduli space}

{\bfseries James W. Anderson, Hugo Parlier, Alexandra Pettet\symbolfootnote[2]{\normalsize Research of 2nd author supported by Swiss National Science Foundation grant number PP00P2\textunderscore 128557,
Research of the 3rd author was partially supported by NSF grant DMS-0856143, EPSRC grant EP/D073626/2, and an NSERC Discovery Grant. 
} }

{\em Abstract.} A closed hyperbolic surface of genus $g\ge 2$ can be decomposed into pairs of pants along shortest closed geodesics and if these curves are sufficiently short (and with lengths uniformly bounded away from $0$), then the geometry of the surface is essentially determined by the combinatorics of the pants decomposition. These combinatorics are determined by a trivalent graph, so we call such surfaces {\em trivalent}. 

In this paper, in a first attempt to understand the ``shape'' of  the subset $\ts$ of moduli space consisting  of surfaces whose systoles fill, we compare it metrically, asymptotically in g, with the set $\tri$ of trivalent surfaces. As our main result, we find that the set $\ts \cap \tri$ is metrically ``sparse'' in $\ts$ (where we equip $\moduli$ with either the Thurston or the Teichm\"uller metric).

\vspace{1cm}

\section{Introduction} \label{introduction}

Although there exists a rich theory describing the geometry of the moduli space $\moduli$ of a closed orientable surface of genus $g \geq 2$, surprisingly little is understood about qualities that might be described as aspects of its ``shape'' when equipped with a natural metric. Such properties are especially meaningful for various subsets of $\moduli$, such as its so-called $\varepsilon$-{\it thick part} $\moduliep$ consisting of those surfaces with injectivity radius bounded from below by some fixed $\varepsilon>0$. 

Rafi and Tao \cite{Rafi-Tao} have offered a starting point for these investigations, proving that 
$$\diam(\moduliep) \asymp \log (g)
$$
when $\moduli$ is equipped with either the Thurston or the Teichm\"uller metric. (Here $\asymp$ means equal up to multiplicative constants that do not depend on $g$.) Their strategy is to metrically approximate $\moduliep$ by the subset $\tri$ consisting of trivalent surfaces: such surfaces are surfaces with a pants decomposition with all curves of lengths bounded above and below by positive constants independent of $g$. 

A further motivation for our study comes from the well-known and difficult problem of constructing a spine for moduli space, i.e., a deformation retract in $\moduli$ of minimal dimension. Equivalently, one would like to find a mapping class group equivariant deformation retract of minimal dimension in Teichm\"uller space $\teich$. In a short preprint \cite{Thurston-spine}, Thurston proposed as a candidate spine the subset in $\moduli$ consisting of those hyperbolic surfaces whose {\it systoles} fill the surface. (A systole of a hyperbolic surface is a non-trivial geodesic of minimal length.) Thurston provided a sketch of a proof that this set is a deformation retract, which unfortunately appears difficult to complete. In particular the contractibility and the connectivity of his candidate remain open. It is moreover not clear how to determine the dimension of this set; for more on this, see \cite{Anderson-Parlier-Pettet}. We will refer to the set in moduli space of those hyperbolic surfaces whose systoles fill as the {\it Thurston well-rounded retract}, or simply the {\it Thurston set}, denoted by $\ts$.

In this paper, we make a first attempt at understanding the ``shape'' of the set $\ts$ by means of comparing it to the set $\tri$ of trivalent surfaces whose shape is becoming well-understood. We are specifically interested in the extent to which the subsets $\ts$ and $\tri$ of $\moduli$ (or the parts that lie in the thick part $\moduliep$ of $\moduli$) imitate each other metrically, asymptotically in $g$, with either the Thurston (or Lipschitz) metric or the Teichm\"uller metric on $\moduli$.  We first note that Rafi and Tao's result gives an obvious upper bound on the diameter of $\ts$ since it is a compact subset of  $\moduliep$ for sufficiently small  $\varepsilon$. A lower bound, matching asymptotically this upper bound, is implicit in examples that arise naturally from our results here. Our first main result is the following.

\begin{theorem}\label{th:BigBers}
There exists a sequence of surfaces $S_{g_k}$ of genus $g_k\to \infty$ with a filling set of systoles and with Bers constant $>2 \sqrt{g_k}$.
\end{theorem}

Our examples, as are some of those of Rafi and Tao, are based on the so-called ``hairy torus" examples due to Buser \cite{Buser}. For the lower bounds on their diameter estimates, Rafi and Tao also use the different shapes one can produce with a trivalent graph, imitating these shapes with trivalent surfaces. Similarly, we are able to do this in $\ts$. Specifically we obtain the following, where ${\bf X}_g^\ell$ is the subset of $\ts$ where the systole is equal to $\ell$. 

\begin{theorem}\label{th:QItrivalent}
There exist absolute constants $A$, $B, \ell >0$ such that for any finite trivalent graph $G$ there exist $g\ge 2$ and $S\in {\bf X}_g^\ell$ such that $S$ and $G$ are $(A,B)$-quasi-isometric.
\end{theorem}

From either of these theorems (using the well known relationship between the Thurston and Teichm\"uller metrics \cite{Kerckhoff,Thurston-metric,Wolpert}, it is easy to deduce the following:
\begin{equation*}
\limsup_{g\to \infty} d_H(\ts, \tri) \asymp \log (g),
\end{equation*}
where $d_H(\cdot, \cdot)$ is the Hausdorff distance on $\moduliep$ induced by either the Teichm\"uller or Thurston metrics.  It is an artifact of our constructions that we obtain $\limsup$ rather than a lower bound of order $\log (g)$, and one should not read too much into it. Constructing a surface with a filling set of systoles is somewhat delicate as any small deformation potentially kills all systoles (but one). It seems likely that a more delicate construction than the one give below would provide examples that appear in every genus.

However, it is significant that in Theorem \ref{th:QItrivalent} we do not have strict control over the genus of the surface obtained from an arbitrary trivalent graph. In fact, our next main result proves that, at least within the set $\tri$, the surfaces in $\ts$ are in some sense sparse. 

\begin{theorem}\label{thm:main}
There exists an absolute constant $C>0$ such that a {\it random} surface in $\tri$ has distance in $\tri$ at least $C \log (g)$ from $\ts$.
\end{theorem}

In particular, this would be in contradiction with Theorem \ref{th:QItrivalent} had the control on genus been too strict. 

Note that here the distance we consider on $\tri$ is the path metric obtained by computing path lengths in $\tri$ with either the Thurston or Teichm\"uller metrics (and thus {\it not} the restriction of the distance functions to $\tri$). A more detailed discussion of this can be found in Section \ref{sec:counting} where we also make the randomness in the statement precise. Loosely speaking by ``random" we mean any reasonable notion of random coming from the natural counting measure associated to the trivalent graphs used to construct the surfaces of $\tri$. The techniques used to proved this theorem seem to have independent interest, involving a number of lemmas about pants decompositions and graph counting. 

\section{A surface in the Thurston set with large Bers constant}\label{sec:bers}

Recall that a {\it pants decomposition} of a closed hyperbolic surface of genus $g \geq 2$ corresponds to a maximal collection of disjoint  simple closed geodesics, or {\it pants curves}, on the surface. Any connected component of the complement of the curves is a three-holed sphere, or {\it pair of pants}. It is a theorem of Bers, quantified by Buser (see \cite{Buser} and references therein), that every surface has a pants decomposition of length bounded by a function which only depends on the topology of the surface. For us, the {\it Bers constant} of the surface is the smallest $L>0$ for which there is a pants decomposition all of whose curves have length at most $L$. By the theorems cited above, it is known that the Bers constant of any genus $g$ surface is at most $21(g-1)$. On the other hand, there exist surfaces where any pants decomposition has a curve of length at least $\sqrt{6g}-2$. 

In this section, we find surfaces in the Thurston set all of whose pants decompositions have a curve of length at least the square root of their genus (for arbitrarily high genus) 
but with their systoles bounded by an absolute constant. It is perhaps somewhat surprising that a surface has simultaneously a filling set of short curves and only long pants decompositions but in fact Buser's hairy torus examples \cite{Buser} have similar properties. Those surfaces have long pants decompositions with injectivity radius uniformly bounded from above and our family of surfaces is directly inspired by these hairy tori. 

{\bf Theorem \ref{th:BigBers}.} {\em There exists a sequence of surfaces $S_{g_k}$ of genus $g_k\to \infty$ with a filling set of systoles and with Bers constant $>2 \sqrt{g_k}$.}

Before proceeding with the proof, we make a couple of remarks. The first is that, although our proof of Theorem \ref{th:BigBers} provides a slightly better constant in the lower bound, for simplicity we leave this as ``2'' since it is really only the order of growth we are concerned with. Second, any Lipschitz map from a trivalent surface (where the lengths of the curves of some pants decomposition lie in the interval $[a,b]$, with $a$ and $b$ independent of $g$) to $S_{g_k}$ must then have Lipschitz constant comparable to at least $\sqrt{g_k}$, because some pants curve on the trivalent surface must get stretched to a curve of at least that length. Thus Theorem \ref{th:BigBers} implies a lower bound on the Hausdorff distance between $\ts$ and $\tri$, in either the Teichm\"uller or the Thurston metric on $\moduli$. In particular:
\begin{equation*}
\limsup_{g\to \infty} d_H(\ts, \tri) \asymp \log (g).
\end{equation*}

\begin{proof}

We will construct a surface of genus $1$ with $2g-2$ cone points of angle $\pi$, and then view it as the quotient of a genus $g$ surface by an orientation preserving involution with $2g-2$ fixed points. We will then show that this double covering surface has the properties we require.

The basic building block for our surface is the unique hyperbolic quadrilateral with all four angles equal to $\pi/4$ and an order four cyclic isometry. For future reference we refer to this quadrilateral as the {\it square}. It is itself obtained by gluing together four copies of a symmetric Lambert quadrilateral whose only angle not equal to $\frac{\pi}{2}$ is $\frac{\pi}{4}$. If we denote by $t$ the length of the two sides of equal length opposite to the angle $\frac{\pi}{4}$, then $t$ satisfies 
$$
\sinh^2(t) = \cos\left( \frac{\pi}{4}\right)
$$
and thus 
$$
t = \arcsinh \left(\sfrac{1}{2^{\frac{1}{4}}}\right).
$$
What we  need in what follows is that the shortest distance between two opposite sides of the square is $2t$.

Now for every pair of integers $m,n$, we  construct a torus with $m n$ singular points of angle $\pi$ as follows: We arrange $mn$ copies of the square in a $m \times n$ rectangular grid. We then paste opposite sides of the rectangle in the obvious way to obtain a torus $T$. 

Note that the singular points of the torus thus obtained are meeting points of exactly four squares, and so have angle $\pi$. Furthermore there is a natural action of $\mathbb{Z}_m\times \mathbb{Z}_n$ on $T$, and this action is transitive on the squares.

Our goal is to construct large genus surfaces, and so we  have in mind that $m,n$ should be large numbers. If $mn$ is even, there is a closed hyperbolic surface $S$ of genus $g=\frac{mn+2}{2}$ and an orientation preserving involution $\sigma$ of $S$ with $2g-2$ fixed points such that the quotient of $S$ by $\sigma$ is the torus $T$. We must show that such a surface $S$ has a filling set of systoles, and that for a suitable choice of $m$ and $n$, any pants decomposition has a ``long'' curve.

A first observation is that simple geodesic paths on $T$ between pairs of singular points on $T$ (that do not pass through any other singular points than those at their endpoints) lift to simple closed geodesics on $S$. This is because on $S$, the involution $\sigma$ acts as rotation by angle $\pi$ around any fixed point of $\sigma$. As such, each such path lifts to two distinct copies of the path on $T$ between the two fixed points; these copies of the lifted path meet at an angle of $\pi$ at both endpoints and so they form a simple closed geodesic.

Among all pairs of singular points on $T$, there are some pairs that are distinguished by being the singular points that lie at the ends of a side of a single square in $T$.  We refer to these as {\em basic pairs} of singular points.  The paths between basic pairs of singular points are  the shortest paths between distinct singular points on $T$. (To see this one can observe that maximal radius disjoint balls around the singular points meet on the centers of the sides of the square.) 

\begin{figure}
  \centering
  \includegraphics[width=5cm]{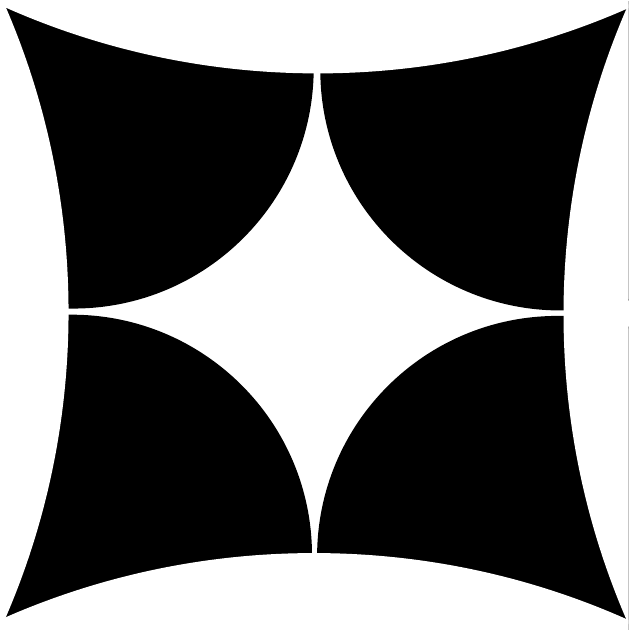}
  \caption{\footnotesize The ``square"}
  \label{fig:quadrilateral}
\end{figure}

In particular, we can use this to show that a shortest path between a basic pair of singular points will lift to a systole of $S$. To see this, take a simple closed geodesic $\gamma$ on $S$ and consider its image under the quotient of $S$ by $\sigma$. There are two cases to consider.

\begin{figure}
  \centering
  \includegraphics[width=8.5cm]{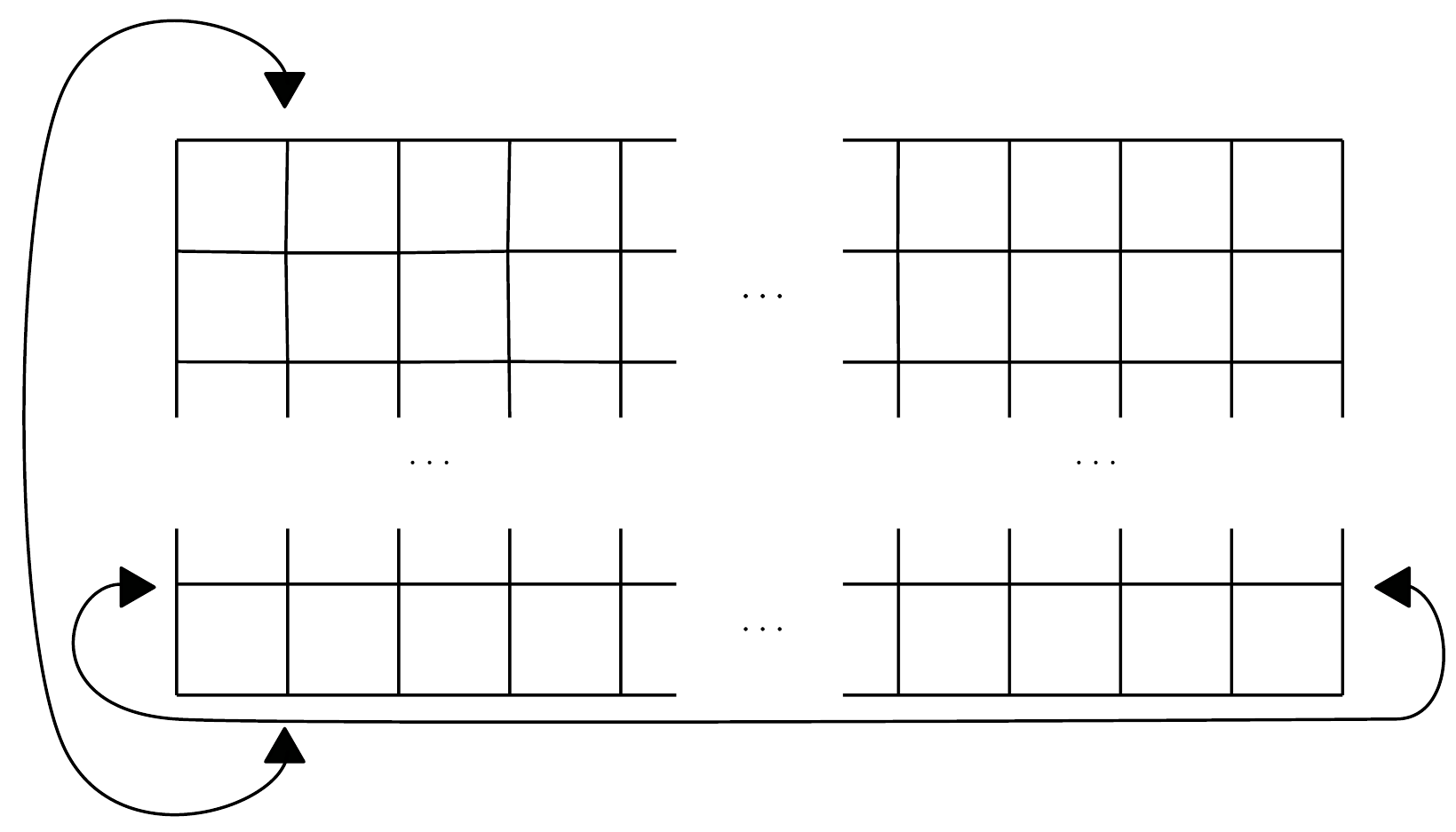}
  \caption{\footnotesize The schematics for building the torus from the quadrilaterals}
  \label{fig:torusgluing}
\end{figure}

If $\gamma$ goes through a fixed point of $\sigma$, then it will be invariant under the action of $\sigma$. Then $\gamma$ contains a second fixed point, necessarily the point diametrically opposite on $\gamma$ to the first fixed point. Thus $\gamma$ will descend to a simple path on $T$ between distinct singular points.

If $\gamma$ is not invariant under the action of $\sigma$ then it must descend to a non-trivial geodesic of the same length on $T$, possibly non-simple. We can conclude that for large enough $m,n$, the shortest closed geodesics on $T$ are exactly those that follow a side of a square between two singular points and then return. 

We now prove that the systoles of $S$ fill. Their quotients under $\sigma$ clearly fill $T$ as the complementary region is a collection of quadrilaterals. This is actually sufficient. To see this suppose that they did not fill $S$. Then there would be a non-trivial curve on $S$ not intersected by its systoles and thus a non-trivial curve in the quotient not intersected by the systoles; this is impossible.

We now focus our attention to pants decompositions. The crucial observation is the following:

{\it Every pants decomposition on $S$ must include a closed curve that descends to a non-trivial curve on $T$.}

Note that we are viewing $T$ here as a topological torus, as opposed to an orbifold. By a non-trivial curve, we mean a representative of a non-trivial element in the  fundamental group of $T$.

{\it Proof of the observation:} 
The image under the involution $\sigma$ of each curve from a pants decomposition on $S$ is a connected, not necessarily simple, curve on $T$. The image of the full pants decomposition under $\sigma$ must fill $T$. In particular, if $T$ is cut along these curves, the resulting surface has genus $0$. This means that the result of cutting $T$ along at least one of these image curves must be genus $0$. This curve is therefore necessarily non-trivial, which proves the observation.

We now establish a lower bound for the length of such a curve in $T$. If we view $T$ as a rectangle with a standard gluing, any non-trivial curve on $T$ must either go ``east-west'' or ``north-south'' (or both). As such it must have length at least $2 t \min\{m,n\}$. The same holds true of course for the preimage of such a curve under the involution. 

Setting $m=n$, we have $n = \sqrt{2g - 2}$, and the theorem follows.
\end{proof}

\section{Trivalent surfaces with filling sets of systoles}
\label{sec:qi}
In this section, we show that any trivalent graph with the combinatorial metric (in which each edge has length one) coarsely resembles a hyperbolic surface with a filling set of systoles. The strategy is to replace each vertex of $G$ with a copy of a torus $Y$ with three boundary components, and to glue together these copies in a way prescribed by the edges of $G$. We work with these higher complexity building blocks $Y$, rather than pants, to give us room to fill the final surface  with systoles. In order to ensure that our constructed set of curves actually are systoles, we require that our graph $G$ have girth $6$ to begin with. Of course not all trivalent graphs have this property, but as we are only interested in graphs up to (uniform) quasi-isometry, we can use the following.

\begin{lemma}\label{lem:graphqi}
There exist absolute constants $a$, $b>0$ such that any trivalent graph $G$ is $(a,b)$-quasi-isometric to a trivalent graph of girth at least $6$.
\end{lemma}
\begin{proof}Suppose that $G$ is a trivalent graph of girth less than 6. At any cycle of length less than 6, we replace one edge by a segment subdivided into nine segments. We label, in linear order, the 8 interior vertices with the integers $1,\ldots, 8$. Let $H$ be the octagonal graph whose eight vertices are cyclically labeled also by $1,\ldots, 8$. We attach a copy of $H$ at each of these newly created length nine segments of $G$ by attaching each of its eight labeled vertices to one of the labeled vertices of the segment of $G$, according to the pairings: 
$(1,1), (2,4), (3,7), (4,2), (5,5), (6,8), (7,3)$, and $(8,6)$. This produces a graph whose girth is at least $6$.

As these alterations to $G$ consist of local insertions of one of a finite possible number of finite graphs, the resulting graph is quasi-isometric to $G$ by constants which do not depend on the graph.
\end{proof}

We recall the statement we would like to prove. 

{\bf Theorem \ref{th:QItrivalent}.} {\em There exist absolute constants $A$, $B, \ell >0$ such that for any finite trivalent graph $G$ there exist $g\ge 2$ and $S\in {\bf X}_g^\ell$ such that $S$ and $G$ are $(A,B)$-quasi-isometric.}

\begin{proof} Using the previous lemma, we can assume that $G$ has girth at least 6.

Our building block is a surface $Y$ constructed from $12$ copies of a right-angled pentagon with side lengths, cyclically ordered, $s/2, s/6, s/4, b/4, c$. Specifying this relation between side-lengths gives us the following relations between $s$ and $b$:
\begin{eqnarray*} 
\sinh(s/2) \sinh(s/6) = \cosh(b/4) \\
\sinh(s/4) \sinh(b/4) = \cosh(s/2)
\end{eqnarray*}
This system of equations uniquely determines values for $b$ and $s$, namely that 
\begin{eqnarray*}
	s \approx 4.39 \quad \quad \quad
	b \approx 7.77
\end{eqnarray*}
For us, the main significance is that $b > s$. A similar system of equations determines that $c$ is $s/12$. This value of $c$ is what determines our choice to assume that $G$ have girth at least 6, as we shall see below.

\begin{figure}
  \centering
    \includegraphics[width=12cm]{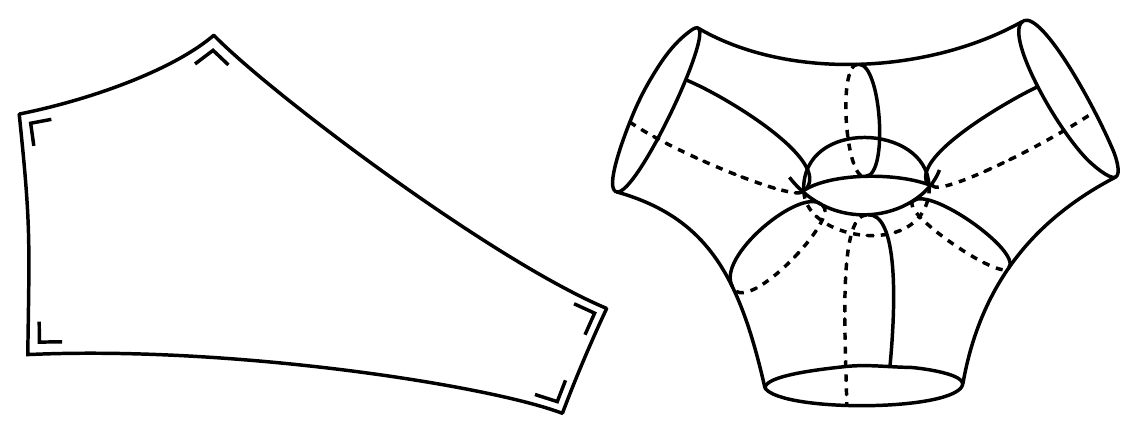}
  \caption{\footnotesize A torus with three boundary components (right) built out of 12 copies of a right-angled pentagon (left).}
  \label{fig:VertexSurface}
\end{figure}

The 12 pentagons are arranged so that $Y$ is a genus one surface with three boundary components, each of length $b$ (see Figure \ref{fig:VertexSurface}). Note that there is a pants decomposition of $Y$ into three pants, each having two boundary components of length $s$ and one of length $b$. There is an order three fixed point free isometry of $Y$ which cyclically permutes these pants.   Since $b > s$, the lengths of the boundary curves of $Y$ are greater than $s$. Whenever we refer to a pair of pants in $Y$, we are referring to the pants in this particular decomposition.  We refer to the copies of $Y$ used to make up the surface as {\em $Y$-pieces}.

\begin{figure}
  \centering
  \includegraphics[width=6.0cm]{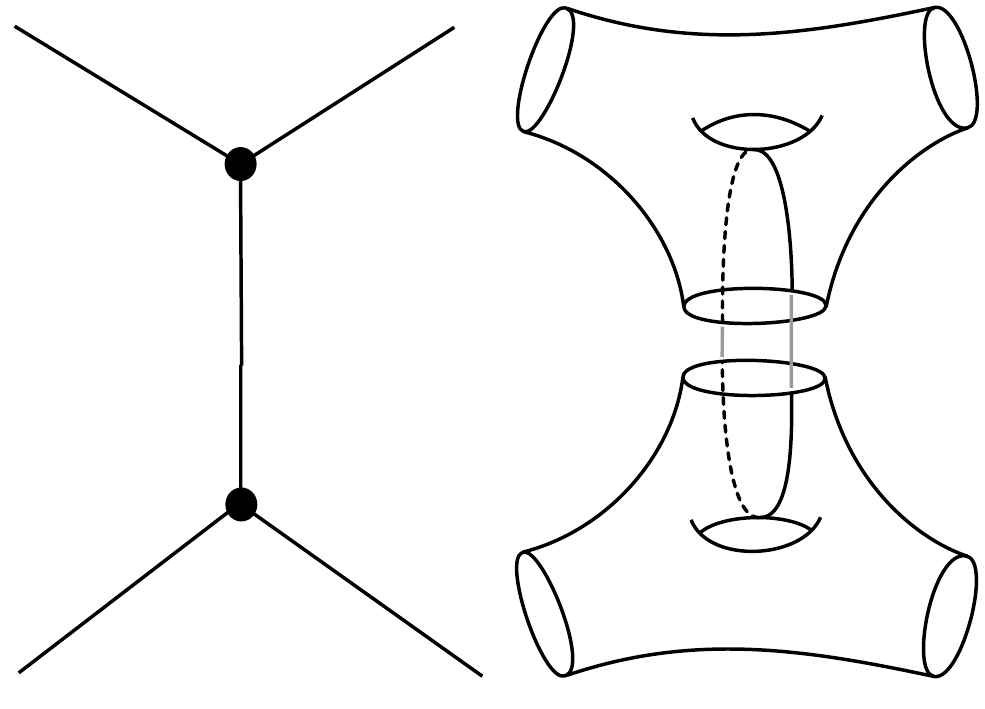}
  \caption{\footnotesize A section of the graph (left) and the corresponding gluing of $Y$-pieces (right).}\label{fig:QIsurface}
\end{figure}

For each vertex $v$, we label the outward directions of the edges from $v$ by $v_1, v_2, v_3$. For the corresponding copy $Y_v$ of the surface $Y$, we likewise label the three boundary components by $v_1, v_2,  v_3$. Now we attach the copies of $Y$ (referred to hereafter as {\em $Y$-pieces}), one for each label, as prescribed by the graph: if vertices $v,w$ of the graph are adjacent along directions $v_i$ and $w_j$, then the corresponding $Y$-pieces $Y_v, Y_w$ are glued together without twist along the boundary curves having the same labels. Specifically, by without twist we mean as in Figure \ref{fig:QIsurface}.

In Figure \ref{fig:SystoleSet} we indicate a collection $\mathcal{S}$ of geodesics which we claim to consist of systoles on $S$.  By construction, the curves in $\mathcal{S}$ all have length $s$.  That $\mathcal{S}$ fills $S$ is clear by the construction: the complement of these curves consists of hyperbolic polygons.  We now check systematically that every other simple closed curve on $S$ must have length at least $s$. 

\begin{figure}
  \centering
  \includegraphics[width=7.0cm]{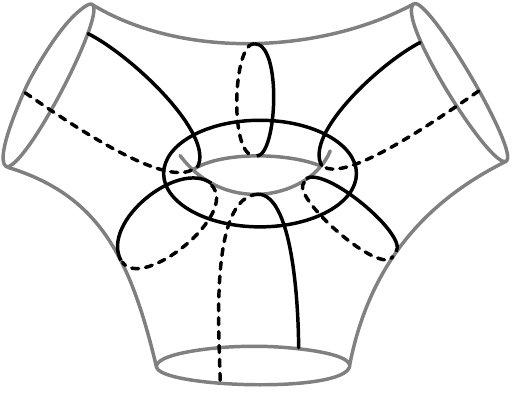}
  \caption{\footnotesize Putative systoles and arcs of systoles on a $Y$-piece, in boldface.}\label{fig:SystoleSet}
\end{figure}

A curve whose projection to $G$ contains a nontrivial cycle must have length at least $s$; this is ensured by the fact that the girth of $G$ is at least 6, together with the fact that the shortest distance between two boundary components of a $Y$-piece is $2c=s/6$. 

We are left to check those curves whose projections to $G$ contain no cycles. In this case, the projection is either a single vertex or a path, which is a subgraph of $G$ homeomorphic to a closed interval.   If the projection of the curve on $S$ to $G$ is a path, then its endpoints lie in distinct $Y$-pieces.

We start by describing some particular geodesic arcs in the $Y$-pieces. We say that a simple arc in a $Y$-piece is of type $\mathcal{O}$ if it lies in a single pair of pants, and if its endpoints lie in a single boundary component of the pants in the interior of the $Y$-piece. An arc is of type $\mathcal{P}$ if intersects exactly once each pair of pants in the $Y$-piece, with its endpoints contained in a single boundary component of the $Y$-piece. An arc is of type $\mathcal{Q}$ if it lies in a single pair of pants and its endpoints lie in the corresponding boundary of the $Y$-piece. Finally, an arc is of type $\mathcal{R}$ if it connects two distinct boundary components of the $Y$-piece. Representative examples of the different types of arcs are illustrated in Figure \ref{fig:Arctypes}.  

We note that there are other possible arcs; for instance, there exists an arc in the $Y$-piece with both endpoints in the same component of the boundary of the $Y$-piece but which intersects two different pairs of pants in the $Y$-piece.  However, such an arc contains an arc of type $\mathcal{O}$ as a subarc.  We have chosen the arcs of types $\mathcal{O}$, $\mathcal{P}$, $\mathcal{Q}$ and $\mathcal{R}$ because every curve on $S$ that projects to a path in $G$ contains a sufficient number of arcs of these types for us to estimate its length. 

\begin{figure}[h]
\leavevmode \SetLabels
\L(.62*.73) $\mathcal{Q}$\\
\L(.41*.8) $\mathcal{O}$\\
\L(.515*.31) $\mathcal{P}$\\
\L(.63*.595) $\mathcal{R}$\\
\endSetLabels
\begin{center}
\AffixLabels{\centerline{\epsfig{file =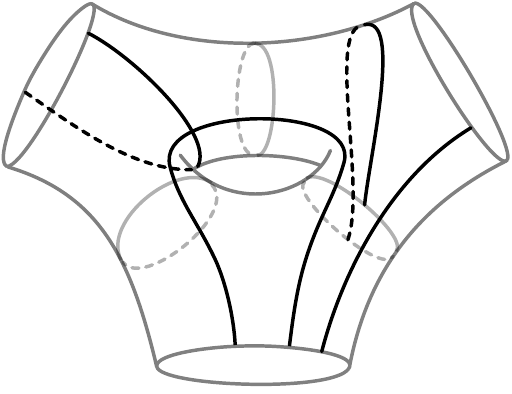,width=7.0cm,angle=0}}}
\vspace{-25pt}
\end{center}
\caption{Arcs on a $Y$-piece, labelled by their type.}\label{fig:Arctypes}
\end{figure}
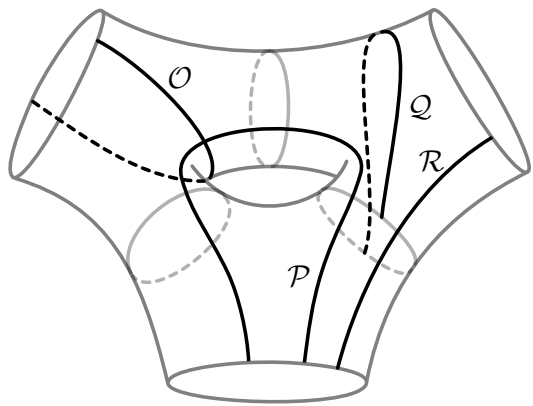

We can bound from below the length of any simple closed curve $\gamma$ on $S$ whose projection to the graph is a path by bounding the length of any arc of type $\mathcal{O}$, $\mathcal{P}, \mathcal{Q}$, or $\mathcal{R}$.
\begin{itemize}
	\item An arc $\alpha$ of type $\mathcal{O}$: Suppose that the ends of $\alpha$ lie on a curve $\gamma$ in $\mathcal{S}$. Each component of $\gamma - \alpha$, together with $\alpha$, forms a simple closed curve. One such curve is homotopic to one of the curves of length $b$, the other homotopic to another curve in $\mathcal{S}$, which necessarily has length $s$. We obtain thus the inequality: 
$$ 2 \ell(\alpha) + s \geq b + s > 2 s$$
so that the length of $\alpha$ satisfies $ \ell(\alpha) > s $.

	\item An arc $\alpha$ of type $\mathcal{P}$: Note that the shortest path connecting the boundary components of length $s$ of a pair of pants has length $s/3$. The arc $\alpha$ has two such subarcs, and so we have 
$$ \ell(\alpha)> s/3 + s/3 = 2s/3 $$

	\item An arc $\alpha$ of type $\mathcal{Q}$: That $\ell(\alpha) \geq s/2$ follows immediately from the construction.
	\item An arc $\alpha$ of type $\mathcal{R}$: Such an arc must have length at least $s/6$. Note here that because the girth of the graph $G$ is at least 6, any curve containing an arc of type $\mathcal{R}$ must have length at least $s$.
\end{itemize}

If $\gamma$ is not properly contained in a $Y$-piece, then it contains at least two arcs each whose lengths are at least that of any arc from types $\mathcal{O}$, $\mathcal{P}, \mathcal{Q}$, or $\mathcal{R}$. Thus $\gamma$ in this case  has length at least $s$.

The remaining case is that of a simple closed curve $\gamma$  properly contained in one of the $Y$-pieces. If $\gamma$ intersects each pair of pants along arcs which connect distinct boundary components of the pants, then the length of $\gamma$ is bounded below by $s$. Otherwise $\gamma$ intersects pants along two disjoint simple arcs 
whose length is at least that of an arc of type $\mathcal{O}$. Then again, the length of $\gamma$ is at least $s$. 

We have thus established that any essential curve on $S$ has length at least $s$, and thus that the set of curves $\mathcal{S}$ consists of systoles. 
\end{proof}

We close this section by remarking that we can loosen the requirement that the valence of $G$ be exactly 3. Rather, we can stipulate that the valences of the vertices be uniformly bounded. A graph with bounded vertex valence is quasi-isometric to a trivalent graph, with quasi-isometry constants depending only on the bounds on the valences.
\section{$\ts$ is sparse in $\tri$}\label{sec:counting}

Recall that we define the set $\tri={\bf Y}_g^{[a,b]}$ of trivalent surfaces with bounded length pants decomposition as 
$${\bf Y}_g^{[a,b]} =\{\text{surfaces with a pants decomposition } \{\gamma_k\} {\text{ with }} \ell(\gamma_k) \in [a,b] \}$$

In the previous section, we found that any trivalent graph (and thus any surface in $\tri$) can be imitated, up to quasi-isometry with absolute constants, by a surface in the Thurston set $\mathbf{X}_{g'}$ where $g$ and $g'$ are comparable (but are not necessarily equal). In this section we show something somewhat tangent to this result: namely that $\tri \cap \ts$ is in some sense sparse inside of $\tri$. Specifically we investigate how well distributed points of $\ts$ are in $\tri$ as the genus $g$ increases.

In order to endow $\tri$ with a geometry we begin with the following observation.

{\it
There exist constants $0 < a < 2 \arcsinh (1) < b$ such that $\tri = {\bf Y}_g^{[a,b]}$ is a  path-connected subset of $\moduli$.}

The existence of such constants comes from the fact that the pants graph is connected and that elementary moves can be implemented for well chosen constants. As they are not terribly important in what follows we do not dwell on optimizing them. A simple computation, for instance, shows that $a=\frac{1}{10}$ and $b=10$ suffice.

As a path-connected subspace, $\tri$ can be equipped with the path metric coming from the induced metric we are considering on Teichm\"uller space. We focus our interest on two metrics on this space: the  Teichm\"uller metric and the Thurston metric. What follows holds for either choice, so let us suppose that we have chosen one of the two on Teichm\"uller space and denote the induced distance on $\tri$ by $d_{\tri}$. Note that if we knew $\tri$ were sufficiently convex with respect to either metric, this path metric $d_{\tri}$ would be close to the restricted metrics on Teichm\"uller space. However, little is known about the convexity of this set.

The metric structure we consider on $\tri$ admits a nice combinatorial description, in terms of the quotient $\MDP(\Sigma_g)$ of the {\em diagonal pants graph} of a closed orientable surface $\Sigma_g$ of genus $g\ge 2$ by the mapping class group. Rafi and Tao \cite{Rafi-Tao} refer to this set as the set of trivalent graphs endowed with the metric of simultaneous Whitehead moves. It can be viewed as the graph of isomorphism types of trivalent graphs of fixed size where two trivalent graphs share an edge if they can be related by elementary moves that can be realized simultaneously. We give a precise and alternative way of seeing it below.

We show the following relationship between $\MDP(\Sigma_g)$ and $\tri$.
						
\begin{theorem}\label{thm:qi}
There exist absolute constants $A,B$ such that for any genus $g\ge 2$, there exists an $(A,B)$ quasi-isometry between $\MDP(\Sigma_g)$ and $\tri$.
\end{theorem}

The fact that one can bound distances in $\tri$ by distances in $\MDP(\Sigma_g)$ is an essential ingredient in the work of Rafi and Tao mentioned above. We describe the space $\MDP(\Sigma_g)$ more precisely. The usual pants graph $\mathrm{P}(\Sigma_g)$ is the graph whose vertices correspond to isotopy classes of pants decompositions of $\Sigma_g$; two vertices span an edge if they are related by a so-called {\em elementary move} (see Figure \ref{fig:elementarymoves}).

\begin{figure}
  \centering
  \includegraphics[width=8.5cm]{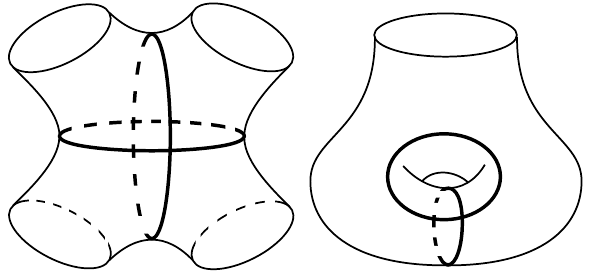}
  \caption{\footnotesize The two types of elementary moves between pants decompositions}
  \label{fig:elementarymoves}
\end{figure}

An elementary move is an exchange of a curve of a pants decomposition with another lying on the same complexity $1$ subsurface and which intersects the original curve minimally. (The {\em complexity} of a surface with boundary is the number of interior curves in a pants decomposition so a complexity $1$ subsurface is either a four holed sphere or a one holed torus.)

The {\it diagonal} pants graph $\mathrm{DP}(\Sigma_g)$ is obtained from $\mathrm{P}(\Sigma_g)$ by adding edges  whenever two pants decompositions are related by elementary moves that take place on disjoint complexity $1$ subsurfaces. The space $\MDP(\Sigma_g)$ is the quotient of $\mathrm{DP}(\Sigma_g)$ by the natural action of the (full) mapping class group. The result is a finite graph whose vertices are the topological types of pants decompositions (as opposed to isotopy classes), where two vertices span an edge if they are related by a set of elementary moves that can be simultaneously realized. All these spaces are naturally metric spaces when one assigns length $1$ to each edge. The space $\mathrm{DP}(\Sigma_g)$ is defined as above in \cite{Aramayona-Lecuire-Parlier-Shackleton}.

With this in hand, we now relate the set $\tri$ to $\ts$ using the induced distance $d_{\tri}$.

{\bf Theorem \ref{thm:main}.} {\em There exists an absolute constant $C>0$ such that a {\it random} surface in $\tri$ has distance in $\tri$ at least $C \log (g)$ from $\ts$.}

For the proof of Theorem \ref{thm:main}, we construct a discrete model that will imitate the geometry of $\tri$. In particular every point in $\tri$ is a fixed distance away from this discrete set and the discrete set is well distributed in $\tri$, i.e., it forms a net. This set as we shall see is a copy of the set of trivalent graphs with $2g-2$ vertices, or alternatively $\MDP(\Sigma_g)$, and so in particular it admits a counting measure. Although there are alternative ways of defining the notion of randomness, the ``random" in the above statement can be taken to be exactly the one coming from this measure. 

Let us now define the discrete set. To each topological type of pants decomposition we associate a surface in $\tri$ equipped with such a pants decomposition. We begin by choosing a constant $\varepsilon_0$. Following a choice of Fenchel-Nielsen coordinates associated to the pants decomposition, we set all pants curve lengths to be $\varepsilon_0$ and all twist parameters to be $0$. If $\varepsilon_0$ is small enough, these pants curves are necessarily systoles. A convenient value for $\varepsilon_0$ is $2 \arcsinh (1)$ (or smaller). Indeed, any curve that crosses a curve of length $\varepsilon_0$ for such a value of $\varepsilon_0$ is necessarily longer by the collar lemma. As such, a surface can have only one such pants decomposition and the pants curves are all systoles. Thus such values of $\varepsilon_0$ determine distinct points in moduli space for distinct topological types of pants decomposition. 

We select in this way a natural discrete subset $\trio$ of $\tri$ consisting of points in one-to-one correspondence with the vertices of $\MDP(\Sigma_g)$. A key step in our proof of the above theorem will establish that the set of points of $\trio$ forms a metric net for $\tri$. 

\subsection{The diagonal pants graph and trivalent surfaces}

In this section we seek to better understand the geometry of $\tri$ using the diagonal pants graph. Our main result is Theorem \ref{thm:qi}. We recall that by $\tri$ we mean ${\bf Y}^{[a,b]}_g$, where $0<a<\varepsilon_0 = 2 \arcsinh (1) <b$ are constants such that ${\bf Y}^{[a,b]}_g$ is path-connected. The constants that arise in what follows will depend on this initial choice of $a$ and $b$, but the essential point is that once $a$ and $b$ are chosen, all subsequent constants are fixed, independent of genus. In all that follows, we are thinking of $\tri$ as a metric space equipped with the induced path metric $d_{\tri}$.

We begin with the following lemma.

\begin{lemma}\label{lem:net} There exists an absolute constant $R>0$ such that any point of $\tri$ lies in a ball of radius $R$ around a point in $\trio$.
\end{lemma}

\begin{proof}
We construct a path of uniformly bounded length entirely contained in $\tri$ between any point in $\tri$ and an appropriately chosen point in $\trio$. We use the following fact which follows from Brooks' theorem in graph coloring.

{\it For every $n\geq 2$ there exists a constant $k_n$ such that the vertices of any trivalent graph can be colored by $k_n$ colors in such a way that any two vertices of the same color are at least distance $n$ apart.}

From this we can deduce the following.

{\underline{Observation:}} {\it A pants decomposition $P$ can be partitioned into at most $k_2$ sets $\Gamma_i$, $i=1,\hdots, k_2$, such that any two curves in the same $\Gamma_i$ lie on distinct pairs of pants.}

For a chosen point in $\tri$, consider a pants decomposition $P$ with curve lengths $\ell_k\in [a,b]$, $k=1,\hdots,3g-3$, and apply the above observation.
For each curve $\gamma\in \Gamma_i$, the union of all pairs of pants it belongs to is a subsurface (either a four holed sphere or a one holed torus). We shall now deform this subsurface in its moduli space while keeping the lengths of each of the curves in its boundary fixed. Specifically, if the Fenchel-Nielsen coordinates with respect to the pants decomposition $P$ of the curve $\gamma$ are $(\ell_k,t_k)$, we use an efficient path in Fenchel-Nielsen coordinates that terminates at $(\varepsilon_0, 0)$.  Here, by an {\em efficient path}, we mean a path of minimal length that remains in the $\varepsilon$-thick part of moduli space. 

We apply this procedure simultaneously to all curves in $\Gamma_i$ (while keeping the lengths and twist parameters of the curves that are not in $\Gamma_i$ fixed). We then proceed to the following set $\Gamma_{i+1}$ and perform the same operation. Starting the procedure with $\Gamma_1$ and terminating with $\Gamma_{k_2}$, we see that the concatenation is the desired path ending at a point at $\trio$. 

Observe that there is an upper bound $L$ on the length of the paths in the moduli spaces of the complexity $1$ subsurfaces that we described above. This follows from a compactness argument. Specifically, because boundary curves have lengths that vary in $[a,b]$ for each $i$, $1 \leq i \leq k_2$, there is a compact set of such moduli spaces. Furthermore, in each of these moduli spaces, there is a compact set of such paths (the initial point must have length and twist in given intervals).  By compactness, this gives the upper bound $L$ (which depends only on $a$ and $b$) on each of the individual paths. Thus the total length of the path is bounded by $R:=k_2 L$.
 \end{proof}
 
Consider surfaces $y_1,y_2\in \trio$ that, via the canonical correspondence with vertices in $\MDP(\Sigma_g)$,  share an edge in $\MDP(\Sigma_g)$. (In the sequel, we will simply say that $y_1$ and $y_2$ {\it share an edge}.) Observe that in $\MDP(\Sigma_g)$, curves that are related via an elementary move must lie in a four holed sphere and not in a one holed torus, as there is only one type of topological non-trivial simple closed curve in the latter.

The surfaces $y_1$ and $y_2$ are geometrically close in the following sense. Consider a curve $\gamma$ of length $\varepsilon_0$ on $y_1$ that is not of the same length on $y_2$. This curve lies in a four holed sphere with boundary curves of length $\varepsilon_0$ on both $y_1$ and $y_2$. It also intersects a curve $\gamma'$ of length $\varepsilon_0$ on $y_2$. Up to action of the mapping class group, we can take $\gamma$ to be a curve of minimal length on $y_2$ among all curves with these same topological properties. It is easy to see that this minimal length only depends on $\varepsilon_0$; it is computable but of no particular interest, other than being independent of genus. 

With that observation in hand, the same proof for Lemma \ref{lem:net} above shows us that the following holds.

\begin{lemma}\label{lem:qi1} There exists an absolute constant $R>0$ such that for all $g\ge 2$ and all pair of surfaces $y_1$ and $y_2$ in $\trio$ that share an edge, 
$d_\tri(y_1,y_2) <R$.\qed
\end{lemma}

This establishes that distances in $\MDP(\Sigma_g)$ bound distances in $\tri$. For the converse, we present the following lemma.  We prove it here in greater generality than is required, as we feel it may be of  independent interest.  We use $\ii(\alpha, \beta)$ to denote the {\em intersection number} between the curves $\alpha$ and $\beta$.  For a collection $Q$ of curves, we define $\ii(\alpha, Q)$ to be the sum of the $\ii(\alpha, q)$ over $q\in Q$. 

\begin{lemma}\label{lem:pantsdistance}
Given $K>0$ there exists a constant $C_K$ such that for any topological type of surface $\Sigma$, the following holds. Let $P$, $Q$ be pants decompositions of $\Sigma$ for which 
$$\ii(\alpha,Q)\leq K {\text{ and }}\ii(\beta,P)\leq K$$
for all $\alpha\in P$ and for all $\beta\in Q$. Then 
$$d_{\DP(\Sigma)}(P,Q) \leq C_K.$$
\end{lemma}

\begin{proof}
We  begin with a rough description of the proof: close to a curve $\alpha$ in $P$, we swap the curves of $Q$ which intersect $\alpha$ with ``nearby" curves that do not. We can do this process simultaneously on curves of $P$ that are far enough apart. The resulting pants decomposition deviates from $Q$ by some bounded distance $D(K)$ dependent only on $K$. We repeat this process along such subsets of curves in $P$ that intersect curves from $Q$ until we reach $P$ itself; through this process we can find the bound $C_K$.

Formally, the construction is as follows: 

For $\alpha\in P\setminus Q$, we denote by $\Sigma_\alpha$ the smallest subsurface of $\Sigma$ for which, if $\beta\in Q$ and $\ii(\alpha,\beta)\neq 0$, then $\beta\subset \Sigma_\alpha$. Observe that $\partial \Sigma_\alpha \subset Q$ and that the curves of  $Q$ that intersect $\alpha$ form a pants decomposition of $\Sigma_\alpha$. It follows that the complexity of $\Sigma_\alpha$ is bounded by $K$. 

\begin{figure}[h]
\leavevmode \SetLabels
\L(.04*.65) $\alpha$\\
\L(.4*.84) $a$\\
\L(.46*.74) $\pi(a)$\\
\endSetLabels
\begin{center}
\AffixLabels{\centerline{\epsfig{file =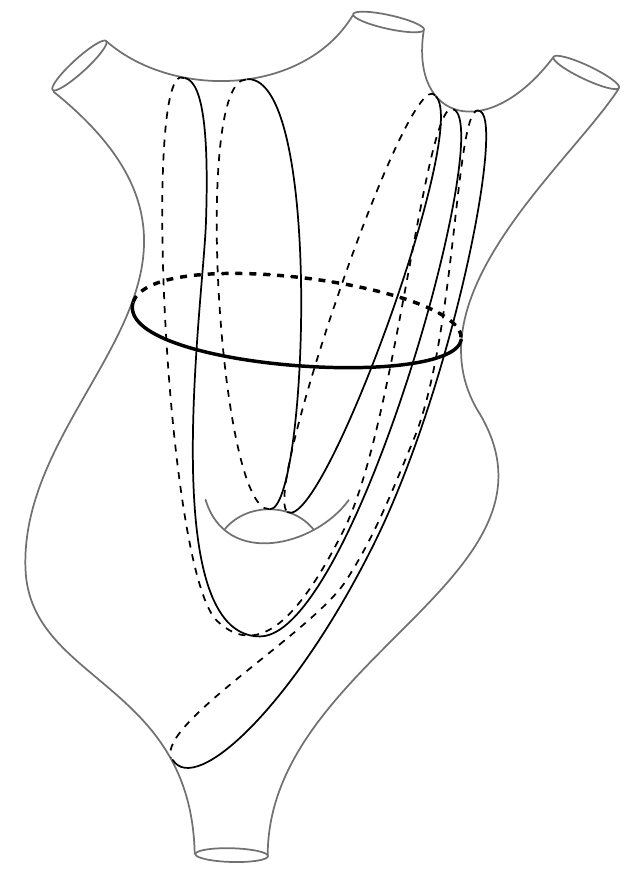,width=5.0cm,angle=0} \epsfig{file =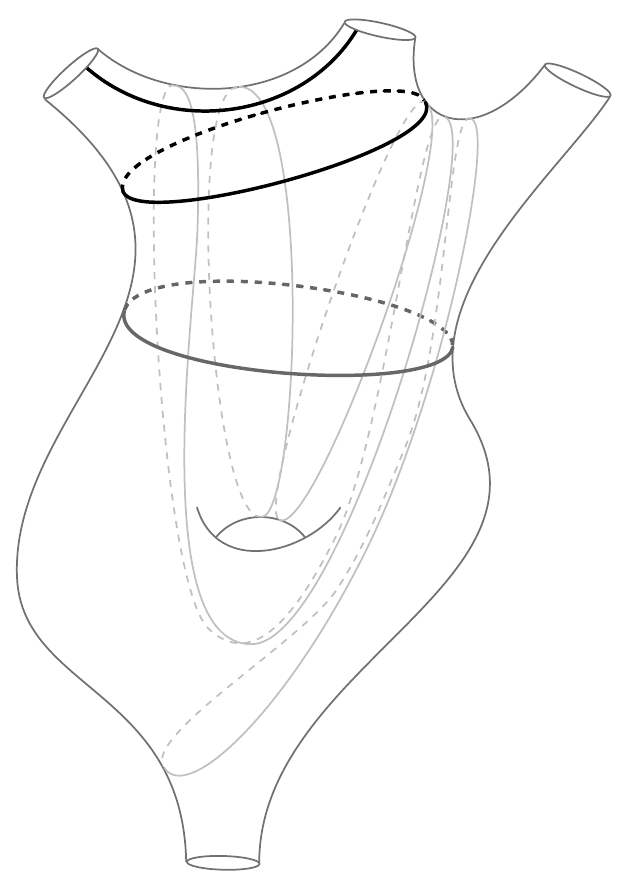,width=5.0cm,angle=0} \epsfig{file =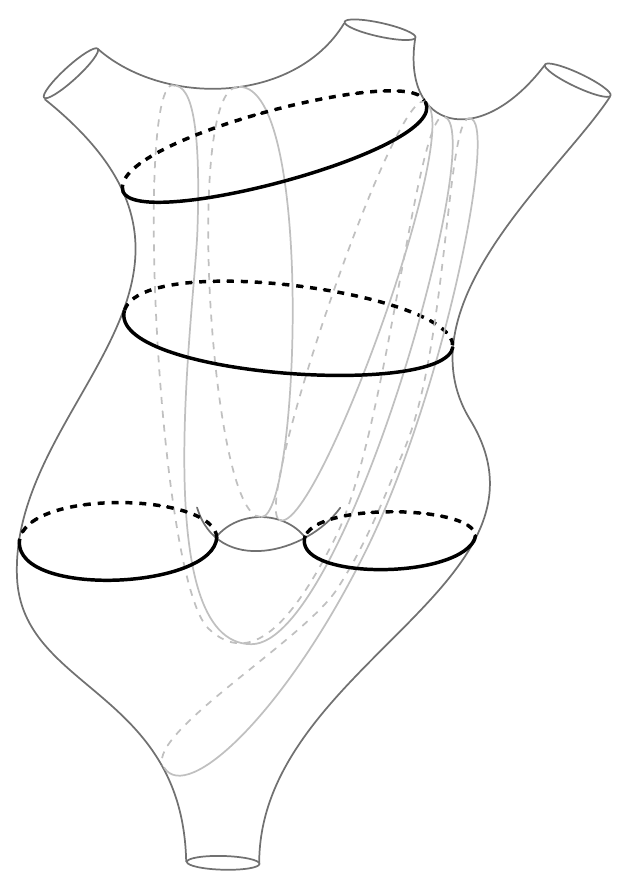,width=5.0cm,angle=0}}}
\vspace{-25pt}
\end{center}
\caption{An example of $\Sigma_\alpha$, $\pi(a)$ and finally $Q_\alpha$}\label{fig:SigmaAlpha}
\end{figure}

Now consider an arc $a$ of a curve of $P$ that essentially intersects $\Sigma_\alpha$. Denote by $\pi(a)$ the subsurface projection of $a$ to $\Sigma_\alpha$: it consists of either one or two curves that, together with the boundary of $\Sigma_\alpha$, bound a pair of pants. Notice that $\ii(\pi(a), \alpha)=0$. 

It is possible that no such arc exists: this means that all boundary curves of $\Sigma_\alpha$ are also curves of $P$. In that event we take a curve $a\in P$ that forms a pair of pants with boundary curves of $\Sigma_\alpha$ and set it to be $\pi(a)$. 

We apply the same process to the subsurface obtained from $\Sigma_\alpha$ by removing the pair of pants bounded by $\partial \Sigma_\alpha$ and $\pi(a)$. Iterating this procedure gives a sequence of pants whose boundary curves, together with $\alpha$, form a pants decomposition of $\Sigma_\alpha$. For future reference we denote $\Sigma_\alpha^k$ be the subsurface obtained by removing the first $k$ pants in this sequence. For instance $\Sigma_\alpha^0 = \Sigma_\alpha$. 

As there are at most $K$ pants in such a decomposition, this process ends in at most $K$ steps by a pants decomposition $Q_\alpha$ of $\Sigma_\alpha$.

We now examine intersection numbers throughout this process (and we won't be very scrupulous with our bounds as our only goal is to show that the bound only depends on $K$). Because the curve containing our initial arc or curve $a$ intersects $Q$ at most $K$ times, we have 
$$
\ii(\pi(a),Q)\leq 2K.
$$
Now suppose that we have an arc $a'$ of a curve in $P$ which intersects $\Sigma_\alpha^1$ essentially. Then we have 
$$
\ii(\pi(a'),Q) \leq 2K + 4K
$$
The $2K$ term bounds intersection coming from the intersection between $a'$ and $Q$. Above we bounded the intersection between $\pi(a)$ and $Q$ by $2K$ and these intersections contribute at most twice to the intersection of $\pi(a')$ with curves of $Q$ (as this projection could wrap around $\pi(a)$). This is the reason for the term $2(2K) = 4K$. Continuing in this way, we find that for any curve $\gamma \in Q_\alpha$
$$
\ii(\gamma,Q) \leq \sum_{i=1}^{K-1} 2^i K
$$

We now compare  $Q_\alpha$  to the restriction of $Q$ to $\Sigma_\alpha$. We claim that the two pants decompositions have a bounded distance $D(K)$ in $\DP(\Sigma_\alpha)$ which depends only on $K$. This follows from a finiteness argument as follows. Using the above estimate and the fact that the complexity of $\Sigma_\alpha$ is bounded by $K$, the total intersection numbers between $Q$ and $Q_\alpha$ are bounded by a function $f(K)$ of $K$ which can be taken to be
$$
f(K) = \sum_{i=1}^{K-1} 2^i K^2.
$$
For $k\leq K$, there are only a finite number of topological types of surfaces of complexity $k$. On each of these types, there are only a finite number of topological types of pants decompositions. For a given pants decomposition, there are only a finite number of isotopy classes of pants decompositions that intersect it at most $f(K)$ times, up to Dehn twists around its curves; this can be seen with a simple constructive argument.

Now we turn this local construction into a global construction and for this we begin by using the observation from the proof of Lemma \ref{lem:net}. Specifically, we partition $Q$ into $m\leq k_K$ multicurves $\Gamma_i$, $i=1,\hdots,m$, such that any two distinct curves in a given $\Gamma_i$ are distance at least $K$ apart on the associated trivalent graph. Note that for distinct $\alpha,\alpha'\in \Gamma_i$, the surfaces $\Sigma_{\alpha}$ and $\Sigma_{\alpha'}$ are disjoint. We begin with $\Gamma_1$ and apply the local construction above to each $\alpha\in \Gamma_1$. The result is a pants decomposition $Q_1$ which contains $\Gamma_1$. As we can make local moves simultaneously, the above discussion for the local construction tells us that there exists a function $D(K)$ depending only on $K$ so that 
$$
d_{\DP(\Sigma)}(Q,Q_1)\leq D(K).
$$
We now apply the same process to $Q_1$, this time with the multicurve $\Gamma_2$. Note that at this stage of the process, the curves of $Q_1$ might have larger intersection number with $P$ than those of $Q$, but as explained above, the function $f(K)$ bounds these intersection numbers. Also notice that curves belonging to $Q_1 \cap P$ are stable during the construction. We now obtain a new pants decomposition $Q_2$ at distance at most $D(f(K))$ from $Q_1$ and we can again repeat the process. At step $i$, we obtain a pants decomposition $Q_i$ whose curves intersect $P$ at most $f^i(K)$ times and at a distance at most $D(f^i(K))$ from $Q_{i-1}$. At step $m$, we obtain $Q_m=P$, and by the above estimates
$$
d_{\DP(\Sigma)}(P,Q) \leq D(K)+D(f(K))+\cdots+D(f^m(K))=:C_K
$$
This proves the lemma.
\end{proof}

As a consequence of this lemma, observe that we have a well-defined projection from $\tri$ to $\MDP(\Sigma_g)$ as follows. For any point $y\in \tri$ we choose a point $y'$ in $\trio$ distance at most $R$ from $y$ (measured using $d_\tri$). This $y'$ - or more precisely the pants decomposition of $\Sigma_g$ corresponding to a a vertex of $\MDP(\Sigma_g)$ obtained via the canonical map from $\trio$ - is the projection of $y$. The existence of such a $y'$ follows from Lemma \ref{lem:net}. There may be more than one such to choose from, but any two are within $d_\tri$ and Teichm\"uller distance (or Thurston distance) $R^2$ from each other. 

Now consider  two such $y_1$ and $y_2$ in $\trio$ at distance at most $R^2$ from one another: we claim that their distance in $\MDP(\Sigma_g)$ is small. Indeed, observe that any short curve $\gamma$ (of length $\varepsilon_0$) of $y_1$ has length at most $\varepsilon_0 e^{R^2}$ on $y_2$, otherwise the two surfaces would be at least $e^{R^2}$ lipschitz apart, a contradiction.  Now by the collar lemma, the curve $\gamma$ intersects the short pants decomposition of $y_2$ at most some function $I(R)$ of $R$ times. By the above lemma we know that $d_{\MDP(\Sigma_g)}(y_1, y_2) \leq C_{I(R)}$. 

Therefore, all choices of $y'$ are the same up to a universally bounded additive error. This allows us to define a map:
 $$
 \Pi:\tri\to \MDP(\Sigma_g).$$
 
We will see that this map $\Pi$ is in fact a quasi-isometry between $\tri$ and its image.  Assuming for the moment that $\Pi$ is indeed a quasi-isometry, then as $\Pi(\tri)$ is uniformly dense in $\MDP(\Sigma_g)$ by the previous lemma, we have proven Theorem \ref{thm:qi}. The fact that $\Pi$ is a quasi-isometry follows from Lemma \ref{lem:qi1} and the following lemma.

\begin{lemma}\label{lem:upperbound}
There exists an absolute constant $C>0$ such that for any $g\ge 2$ and all pants decompositions $P_x$, $P_y$ of $\Sigma_g$ such that $P_x=\Pi(x)$ and $P_y=\Pi(y)$ the following holds:
$$
d_{\MDP(\Sigma_g)}(P_x,P_y) \leq C d_\tri(x,y) + C.
$$
\end{lemma}
\begin{proof}
Let us consider a path $c$ in $\tri$ between $x$ and $y$ of minimum length 
$\ell(c)$ and a covering of $c$ by a minimal number $m$ of balls, $B_1, \ldots, B_m$, of radius $R$, where $R$ is as in Lemma \ref{lem:net}. Note that
$$
m \leq \lfloor \sfrac{\ell(c)}{2 R} \rfloor +1.
$$
We can suppose that $x \in B_1$, $y\in B_m$. For each of the remaining $B_k$, we choose a point $p_k\in B_k\cap \trio$ and project it to the vertex of $\MDP(\Sigma_g)$ via the canonical map $\Pi$. Again, via the same argument as in the discussion preceding this lemma:
$$
d_{\MDP(\Sigma_g)}(P_k,P_{k+1})\leq C_{I(2R)}.
$$
As such
\begin{eqnarray*}
d_{\MDP(\Sigma_g)}(P_x,P_{y}) & \leq & m \, C_{I(2R)}\\
& \leq & C_{I(2R)} ( \lfloor \sfrac{\ell(c)}{2 R} \rfloor +1)\\
& \leq & C \, d_\tri (x,y) + C.
\end{eqnarray*}
\end{proof}

This lemma completes the proof of Theorem \ref{thm:qi}: from $\tri$ to $\MDP(\Sigma_g)$ the quasi-isometry is given by $\Pi$ and from $\MDP(\Sigma_g)$ to $\tri$ we map pants decompositions to $\trio$ via the canonical map. The latter map is quasi-surjective by Lemma \ref{lem:net}. Finally observe that it is the constants from the above lemma that provide $A,B$ in Theorem \ref{thm:qi}.

\subsection{Counting points in $\trio$ at bounded distance from $\ts$}

Having now established a quasi-isometry between $\tri$ and $\MDP(\Sigma_g)$, we will apply it to a counting argument to establish Theorem \ref{thm:main}. 

\begin{notation} For functions $A(x),B(x)$ we will write $A(x)\approx B(x)$ if they are equal up to an exponential factor in $x$.   By this, we mean that there are constants $0 < c_1 < c_2$ so that $(c_1)^x A(x)\le B(x)\le (c_2)^x A(x)$. 

Similarly, we shall use  $A(x)\lesssim B(x)$ if the inequality holds up to an exponential factor.
\end{notation}

The number of vertices in $\MDP(\Sigma_g)$ and thus in $\trio$ is equal to the number of non-isomorphic connected trivalent graphs with $2g-2$ vertices. Quite a bit is known about the number of these graphs, but what we require is the following result of Bollob\'as:

\begin{center}
{\it The number of vertices in $\trio$ is $\approx g^{2g}$.}
\end{center}

The result of Bollob\'as \cite{Bollobas} is actually stronger but this is sufficient for our purposes. Our general strategy will be to establish that for any fixed $a,b>0$ there are considerably fewer points in
$$\Pi(\tri^{[a,b]}\cap \ts)$$
namely $g^{\nu g}$ with $\nu<2$. Of course this result is only interesting if $\tri^{[a,b]}$ and $\ts$ intersect - but as we have seen above, for well chosen constants $a,b$, there are many different shapes of surfaces in $\ts$ which lie in $\tri^{[a,b]}$.

Note that an equivalent result using $\trio$ will also prove to be true: for any $R>0$ there are at most $g^{\nu' g}$ points (with $\nu'<2$) in 
$$\trio \cap B_R(\ts).$$

The result will then follow from the quasi-isometry established in the previous section and an estimate on the growth of balls in $\MDP(\Sigma_g)$.

We begin with the following lemma.

\begin{lemma}\label{lem:short} For any $\ell>0$ there exist constants $L>0$ and $h>0$ such that for any $x\in \ts \cap \tri$ with systole $\ell$, the trivalent graph $\Pi(x) \in \MDP(\Sigma_g)$ contains at least $h g$ disjoint cycles of length at most $L$. 
\end{lemma}

\begin{proof}

We consider a projection of curves on the surface $x$ to paths on the corresponding trivalent graph $\Pi(x)\in \MDP(\Sigma_g)$. For each $x$, the short curves of $x$ will refer to those that belong to the pants decomposition of $\Pi(x)$. To any simple closed curve on $x$, we associate the collection of edges corresponding to the short curves it crosses. Note that the curve might project to a trivial path, by which we mean a subgraph of $\Pi(x)$ without cycles, as do for instance the short curves of $x$.

If the curve projects to a trivial path, this path is necessarily a tree (possibly reduced to a single vertex or a single edge) and because this tree comes from the projection of a closed curve, each edge that appears in this tree must appear an even number of times (that is, the path must go ``back and forth" through each edge it goes through). For this reason, two intersecting curves that both project to trivial paths must intersect at least twice. As such, among any pair of curves that intersect exactly once, at least one projects to a path in the graph containing a non-trivial cycle.

Now consider a homeomorphism $\varphi$ with minimal Lipschitz constant between $x$ and the surface $y\in\trio$ corresponding to the trivalent graph $\Pi(x)$. This Lipschitz constant is at most $e^{R^2}$ (with $R$ from Lemma \ref{lem:net}). We begin by observing that the image under $\varphi$ of a systole $\gamma$ of $x$ is a curve of length at most $e^{R^2}\ell$ on $y$. Now consider the image under $\varphi$ of the full set of systoles of $x$ on $y$. As these curves on $x$ fill, their images under $\varphi$ on $y$ also fill, and hence the projections of these curves from $y$ to $\Pi(x)$ cover every edge. 

As the length of each of the pants curves on $x$ is fixed, the collar lemma ensures that the projection of a systole to $\Pi(x)$ cannot contain more than a certain fixed number $L$ of edges. Now as the systoles fill, every systole intersects another systole (exactly once). As observed above, one of them projects to a non-trivial cycle and as such the $L$-neighborhood of every vertex $v$ in $\Pi(x)$ contains a non-trivial cycle of length at most $L$. 


%
We now establish the lemma by taking a set $V_L$ of vertices of the graph all of distance at least $2L+1$ from each other and by considering cycles of length at most $L$ in the $L$ neighborhood of each. The cycles are all disjoint, and there are at least
$$
\frac{\mbox{total number of vertices}}{\mbox{ number of vertices of distance at most $L$ from a vertex in $V_L$}}
\geq \bigg\lfloor \frac{2g-2}{3 \, 2^{L-1}} \bigg\rfloor
$$

which is bounded below by $h g$ for some suitable choice of $h$. This establishes the lemma.
 \end{proof}
 
We now count trivalent graphs with the property described above. 

\begin{lemma}\label{lem:counting} For constants $h,L,$ there exists a constant $\nu = \nu(h,L)<2$ such that for sufficiently large $g$, there are at most $g^{\nu g}$ isomorphism types of trivalent graphs with $2g-2$ vertices containing at least $hg$ disjoint cycles of length at most $L$.
\end{lemma}
\begin{proof}
The proof is in three steps. We begin by reducing the short cycles to loops by performing a bounded number of simultaneous Whitehead moves on the graph. We then remove the $hg$ loops to obtain a trivalent graph with fewer vertices. Finally we count how many trivalent graphs one can obtain by ``reversing" the above process.

We begin by choosing a set of $hg$ disjoint cycles of length at most $L$. On each such cycle we choose a maximum number of disjoint edges (at most $\frac{L}{2}$). This gives us a collection of disjoint edges of the graph on which we can perform elementary moves simultaneously.

\begin{figure}[h]
  \centering
  \includegraphics[width=14cm]{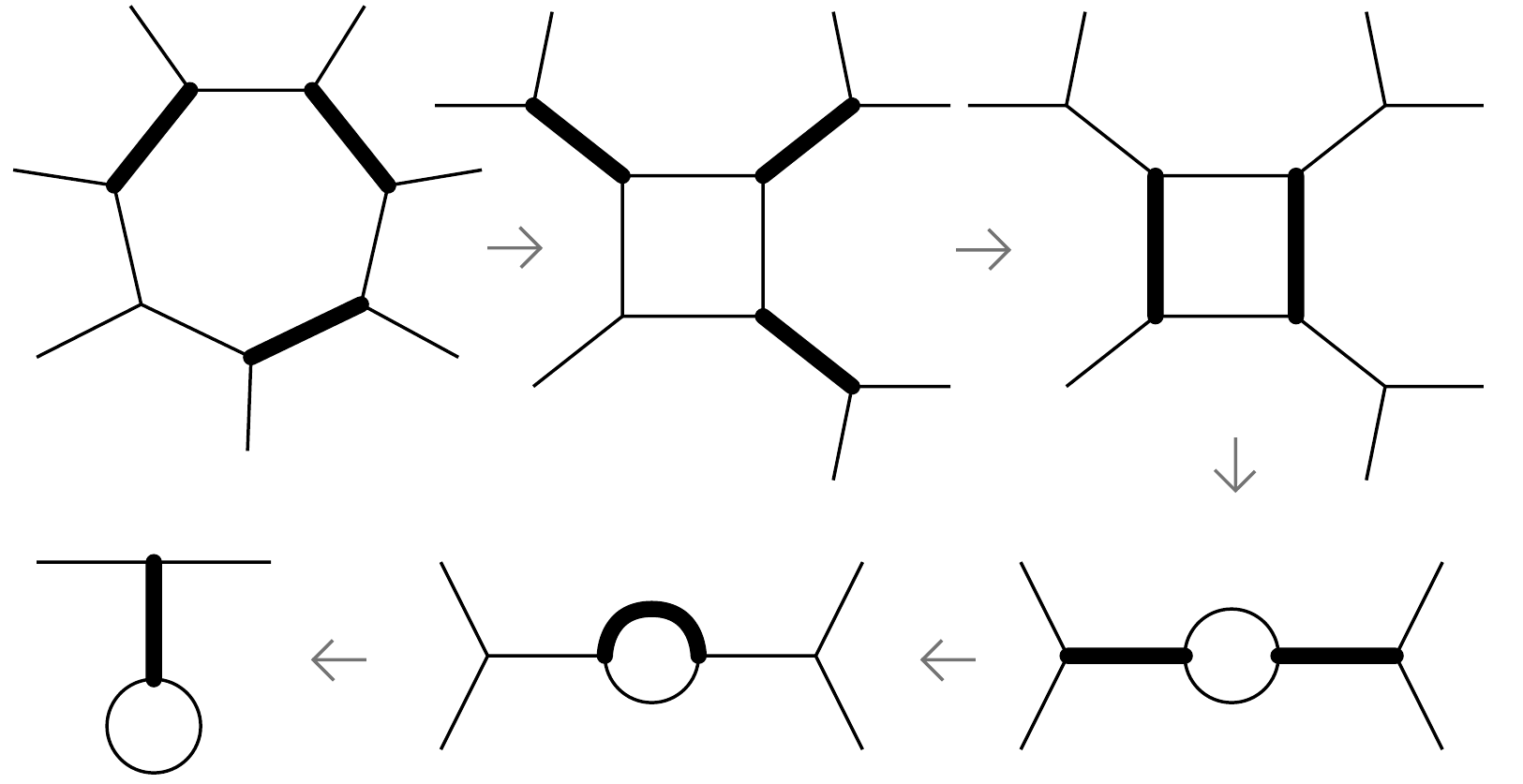}
  \caption{Reducing cycles.}
  \label{fig:Reducing}
\end{figure}

For each loop $\gamma$ with $\ell(\gamma)>1$ we choose the moves carefully so as to reduce the length of the loop by a factor of roughly $\frac{1}{2}$ as in Figure \ref{fig:Reducing}. The resulting loop $\gamma'$ has length precisely 
$$\ell(\gamma')=\lfloor \frac{1}{2} \ell(\gamma) \rfloor +1.$$
Once a cycle has been transformed into a loop, we leave it alone. After at most 
$$\log_2(L) +O(\log_2(L)),$$
steps all $hg$ cycles have been transformed into loops.

Now we consider the trivalent graph obtained by removing the loops as follows: each loop is attached to a vertex which in turn is attached to an edge. We remove the loop and the attached edge and finally we delete the resulting valency $2$ vertices (see Figure \ref{fig:Removing}). The result is a trivalent graph with
$$2g-2-2hg = 2(1-h)g-2$$
vertices. The number of such graphs is at most $\approx$ $g^{2(1-h)g}$, as functions of $g$.
\vspace{0.1cm}
\begin{figure}[h]
  \centering
  \includegraphics{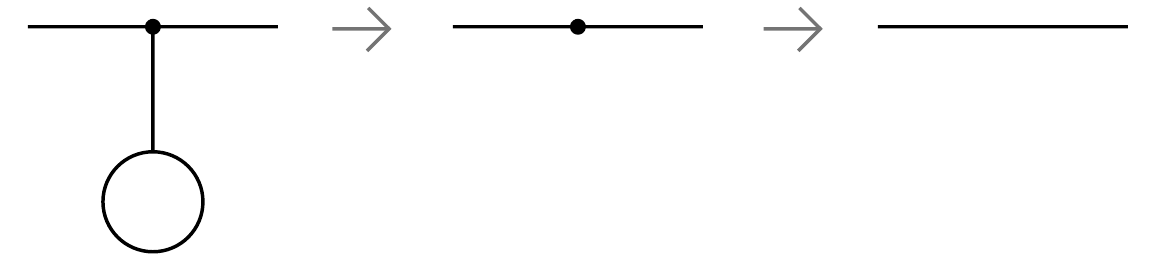}
  \vspace{-.5cm}
  \caption{Removing cycles.}
  \label{fig:Removing}
\end{figure}

This gives a map from graphs with short cycles to trivalent graphs with fewer vertices. By bounding how many different graphs one obtains from a single graph with $2(1-h)g-2$
vertices by reversing the process, we will be able to deduce a bound on how many graphs we began with.

Given a graph with $2(1-h)g-2$ vertices, we need to ``add" vertices back on to it. This consists of adding a vertex to an edge and adding an edge and a loop to the new vertex. A rough bound is given as follows: we may add a loop to an edge or not, and because we have $3(1-h)g-3$ edges, there are at most $2^{3(1-h)g-3}$ possible graphs we can obtain from our initial graph.

Now the above algorithm shows that a trivalent graph with the short cycles is
 in $\MDP(\Sigma_g)$ at most distance 
$$\log_2(L) +O(\log_2(L))$$
away from one of the graphs with $hg$ loops. We can now apply Lemma \ref{lem:balls} (whose proof will follow) to show that there are at most 

$$3^{(g-1)\log_2(L) +O(\log_2(L))}$$

points at distance at most $\log_2(L) +O(\log_2(L))$ from a given point. We now conclude that the total number of graphs with the desired properties is at most 

$$\approx g^{2(1-h)g} 2^{3(1-h)g-3} 3^{(g-1)\log_2(L) +O(\log_2(L))} \approx g^{2(1-h)g}.$$

As such by choosing $2>\nu> 2(1-h)$ we have that this number is bounded by 
$$
g^{\nu g}
$$
for sufficiently large $g$.
\end{proof}

Our last step is to count balls of radius $r$ in $\MDP(\Sigma_g)$.

\begin{lemma}\label{lem:balls} A ball of radius $r$ in $ \MDP(\Sigma_g)$ contains at most $(3^{g-1})^r$ vertices.
\end{lemma}
\begin{proof}
We fix a vertex in $\MDP(\Sigma_g)$ and we begin by counting possible elementary moves that result in producing different pants decompositions and then we bound the possible number of pants decompositions that can result from these moves. An elementary move corresponds to a move across an edge in $\MDP(\Sigma_g)$ only if it occurs on a four holed sphere. There are at most $g-1$ disjoint four holed spheres on a surface and up to homeomorphism an elementary move on a curve in a four holed sphere has three possible resulting curves (including leaving the curve invariant). We deduce that there are at most $3^{g-1}$ different pants decompositions up to homeomorphism that can result from a collection of simultaneous elementary moves. By definition, this is a bound on the number of points in $\MDP(\Sigma_g)$ at distance $1$. At distance $r$, we obtain at most $(3^{g-1})^r$.
\end{proof}

We can now prove Theorem \ref{thm:main}. Via Theorem \ref{thm:qi} and Lemma \ref{lem:short}, it suffices to show that there is a point at least distance $C \, \log(g)$ in $\MDP$ from all elements in $\MDP$ corresponding to graphs with collections of disjoint short cycles. Now via Lemma \ref{lem:counting}, we know the number of such graphs is at most $g^{\nu g}$ with $\nu <2$. Thus via Lemma \ref{lem:balls}, as long as
$$
g^{\nu g} \, (3^{g-1})^r < {\mbox{Number of trivalent graphs with $2g-2$ vertices}}
$$
then there are points of $\MDP$ that are distance $r$ from any point in ``$\trio\cap \ts$". By taking `$\log$'s we have
$$
 r 3^{g-1} \lesssim g^{(2-\nu)g},
 $$
as functions of $g$. As $\nu<2$, from this we obtain the existence of points of distance at least
 $$
 r > C \log (g).
 $$
 This establishes Theorem \ref{thm:main}.

{\it Adresses:}\\
James W. Anderson\\Mathematical Sciences, University of Southampton, Southampton, England\\j.w.anderson@southampton.ac.uk

Hugo Parlier\\Department of Mathematics, University of Fribourg, Fribourg, Switzerland\\hugo.parlier@unifr.ch

Alexandra Pettet\\
Department of Mathematics, University of British Columbia, Vancouver, Canada\\
alexandra@math.ubc.ca

\end{document}